\documentclass{amsart}

\usepackage[utf8]{inputenc}
\usepackage{braket}
\usepackage{amsmath}

\title{Metric Lines in the Special Euclidean group on the plane}

 \author[Wang]{Yuyang Wang}
  \address{Y. Wang;
 {University of Michigan, 530 Church St, Ann Arbor, MI 48109, United States
 \\
\href{wangyy@umich.edu}{wangyy@umich.edu} }
 }

 \author[Ku]{Sean Ku }
  \address{S. Ku;
 {New York University, Courant Institute of Mathematical 
Sciences, 
251 Mercer St, New York, NY 10012, United States
 \\
 \href{sk8980@nyu.edu}{sk8980@nyu.edu}} 
 }

 \author[Bravo-Doddoli]{Alejandro Bravo-Doddoli}
  \address{Alejandro Bravo-Doddoli;
 {University of Michigan, 530 Church St, Ann Arbor, MI 48109, United States
 \\
\href{abravodo@umich.edu}{abravodo@umich.edu}}
 }

\usepackage{caption}
\usepackage{subcaption}
\usepackage[english]{babel}
\usepackage[utf8]{inputenc}
\usepackage{libertine}
\usepackage{graphicx}
\usepackage{floatflt}
\usepackage{blindtext}
\usepackage{enumitem}
\usepackage{amsthm}
\usepackage{listings}
\usepackage{listingsutf8}
\usepackage{framed}
\usepackage{minibox}
\usepackage{float}
\usepackage{wrapfig}
\usepackage{longtable}
\usepackage[strict]{changepage}
\usepackage{pgfplots}
\usepackage{nicefrac}
\usepackage{units}
\usepackage{etoolbox,tikz}
\usepackage{bbm}
\usepackage{hyperref} 

\usepackage[autostyle]{csquotes}

\usetikzlibrary{external}
\usetikzlibrary{cd}

\AtBeginEnvironment{tikzcd}{\tikzexternaldisable}
\AtEndEnvironment{tikzcd}{\tikzexternalenable}

\usepackage{amsfonts}

\usepackage{amsthm}
\usepackage {amssymb}
\usepackage{pb-diagram}
\usepackage{tikz-cd}

\usepackage{anyfontsize}
\usepackage{mathtools}
\usepackage{amssymb}

\usetikzlibrary{matrix}
\pgfplotsset{width=11cm,compat=1.9}
\usepgfplotslibrary{external}
\newtheorem{Theorem}{Theorem}[section]
\newtheorem{defi}[Theorem]{Definition}
\newtheorem{Prop}[Theorem]{Proposition}

\newtheorem{lemma}[Theorem]{Lemma}
\newtheorem{Remark}[Theorem]{Remark}
\theoremstyle{plain}
\numberwithin{figure}{section}
\newtheorem{mainthm}{Theorem}

\newtheorem*{backgroundtheorem}{Background Theorem}

\DeclareMathOperator{\SE2}{SE(2)}
\DeclareMathOperator{\SO2}{SO(2)}
\DeclareMathOperator{\Euclidn}{SE(n)}
\DeclareMathOperator{\Euclid3}{SE(3)}

\newcommand{\sea}{\mathfrak{se}(2)}

\def\R{\mathbb{R}}
\def\D{\mathcal{D}}
\def\Ri{Riemannian }

\def\SR{Sub-Riemannian } 

\def\sR{sub-Riemannian }

\begin{document}

\begin{abstract}
The Special Euclidean group on the plane $\SE2$ has a left-invariant \sR structure. Every \sR manifold possesses a Hamiltonian function governing the \sR geodesic flow. Two natural questions are: What are the necessary conditions for \sR geodesics to be periodic? What type of geodesics are the metric lines in $\SE2$? In this article, we answer both questions, and our method to answering the second is using the Hamilton-Jacobi theory.

\end{abstract}

\maketitle

\tableofcontents

\section{Introduction} 

The Euclidean special group, denoted by $\SE2$, has the structure of a \sR manifold. A \sR geodesic flow is a Hamiltonian system on the cotangent bundle $T^*\SE2$ with the property that a solution projected to $\SE2$ is a \sR geodesic, i.e., locally minimizing horizontal curve. A classic problem in a general Hamiltonian system is the conditions for periodic solutions, and an essential question in \sR geometry is the classification of metric lines. The main goal of this paper is to characterize periodic \sR geodesics and metric lines.

To endow $\SE2$ with a \sR structure, we consider the non-integrable distribution $\D$ framed by the left-invariant vector fields $\{X_{\theta}, X_u \}$, see equation \eqref{eq:left-inv} below, and declare them to be orthonormal. Thus, we obtain a left-invariant \sR metric in $\SE2$, consult \cite{Arnold} or \cite{agrachev} for more details on left-invariant metric. In general, if $\mathbb{G}$ is a Lie group with a left-invariant \sR metric, a standard approach to study the \sR geodesic flow is to consider the symplectic reduction of $T^*G$ by $\mathbb{G}$. In the particular case of $\SE2$, the semidirect group structure of $\SE2$ given by $\SE2 = \SO2 \ltimes \R^2$ plays a primary role, consult \cite{marsden-weinstein} for the general theory of symplectic reductions, and \cite{semi-direct} or \cite[Chapter 4]{marsden2007hamiltonian} for the theory in the case of semidirect products. However, we will use an alternative method; the group $\SE2$ has the structure of a metabelian Carnot group, i.e., the commutator group $[\SE2,\SE2] \simeq \R^2 $ is abelian. Thus, we will consider the Hamiltonian action of $\R^2$ and perform the symplectic reduction of $T^*\SE2$ by $\R^2$,  where the reduced space $T^*\SE2 // \R^2_{\mu}$ is symplectic diffeomorphic to $T^*\SO2$; refer to \cite{ABD-sympletic} for a detailed discussion of symplectic reductions in the case of metabelian groups.

The first consequence of the symplectic reduction by $\R^2$ is the primary tool in our work, a bijection between \sR geodesics in $\SE2$ and curves $\alpha_{\mu}$ in $T^*\SO2$ depending on a parameter $\mu$ in $\R^2$, where we identify $(\R^2)^*$ with $\R^2$. Let us consider the coordinates $(p_{\theta},\theta)$ in $T^*\SO2$, then the curve $\alpha_{\mu}$ is defined by the equation 
\begin{equation}\label{eq:alp-cur}
 \alpha_{\mu} := \{ (p_{\theta},\theta) \in  T^*\SO2 :    1 = p_{\theta}^2 + R^2\cos^2( \theta -\delta) \} ,  
\end{equation}
where $\mu = (R,\delta)$.

In sub-sub-Section \ref{subsub:back-ground}, we provide a prescription to build a curve in $\SE2$ given a curve $\alpha_{\mu}$; the \textbf{Background Theorem} states that the prescription yields \sR geodesic in $\SE2$ parameterized by arc-length. Conversely, the prescription can achieve every arc-length parameterized geodesic in $\SE2$ by applying it to some curve $\alpha_{\mu}$.  

The first main theorem classifies the periodic geodesics over $\SE 2$. 

\begin{mainthm} \label{main:the-periodic}
The corresponding geodesic to $\alpha_{\mu}$ with $\mu = (R,\delta)$ is periodic if and only if $R=0$.  
\end{mainthm}

The proof of Theorem \ref{main:the-periodic} relies on the period map from Proposition \ref{prp:period-map}. 
To introduce our second main result, let us formalize the definition of a metric line.

\begin{defi}
\label{def:metric-line}
 Let $M$ be a \sR manifold, let $dist_{M}(\cdot,\cdot)$ be the \sR distance on $M$, and let $|\cdot|:\R \to [0,\infty)$ be the absolute value. We say that a curve $\gamma:\R \to M$ is a metric line if it is a globally minimizing geodesic, i.e.,  
$$|a-b| = dist_{M}(\gamma(a),\gamma(b))\;\;\; \text{for all compact intervals}\;\; [a,b] \subset \R. $$
\end{defi}
Alternative terms for \enquote{metric lines} are:  \enquote{globally minimizing geodesics}, \enquote{isometric embeddings of the real line}, or \enquote{infinite-geodesics.}

In \cite{RM-ABD,ABD-metric}, A. Bravo-Doddoli and R. Montgomery used the metabelian structure of the Carnot group $J^k(\R,\R)$ to provide a partial result to the classification of metric lines. We will follow their approach. In sub-sub-Section \ref{sub-sec:red-dyn}, we will compute the reduced Hamiltonian $H_{\mu}: T^*\SO2 \to \R$ and classify the \sR geodesics in $\SE2$ according to their reduced dynamics,  this classification is equivalent to classifying the curve $\alpha_{\mu}$, see Lemma \ref{lem:alp-clas} below. Then, we will show that a \sR geodesic is one of the following types: line, $\theta$-periodic, or heteroclinic geodesics. See sub-sub-Section \ref{subsubsec:clas-geo} for formal definitions.

In \cite{Moiseev2010}, I. Moiseev and Y. Sachkov used optimal synthesis to study the cut times of the \sR geodesics of the Special Euclidean group $\SE2$. One of their main contributions was proving that all the \sR geodesics of the type heteroclinic are metric lines. In \cite{Bicycle-paths}, A. Ardentov, G. Bor, E. Le Donne, R. Montgomery, and Y. Sachkov interpreted the \sR geodesics as a bicycle path. They presented an alternative proof to show that the heteroclinic geodesics are metric lines by using an isometry between the geodesic of type lines and heteroclinic. The second goal of this paper is to use Hamilton-Jacobi theory to provide a complete classification of metric lines on $\SE2$ and then give a third proof to I. Moiseev and Y. Sachkov's result.

The second main theorem precisely classifies the metric lines over $\SE 2$. 
\begin{mainthm} \label{main:the-1}
The metric lines in $\SE2$ are precisely the geodesics of the type line and heteroclinic.    
\end{mainthm}

The proof to Theorem \ref{main:the-1} consists of two parts: proving that the geodesics of the type heteroclinic and line are metric lines and showing that the geodesics of the type $\theta$-periodic are not globally minimizing. Our method to prove the first part is to find a calibration function by solving the Hamilton-Jacobi equation. For more details on Hamilton-Jacobi theory, refer to \cite[Section 47]{Arnold} or \cite[Section 47]{Landau}. A calibration function is a tool from weak KAM theory. Consult \cite{fathi2008weak,aub-Ham-eq} for the general weak KAM theory, refer to \cite[sub-sub-Chapter 1.9.2]{tour}, and to \cite[Section 5]{RM-ABD} for more details of the calibration function in the context of \sR geometry. The second part of Theorem \ref{main:the-1} is a consequence of Proposition \ref{prop:cut-time}, which states that geodesics of type $\theta$-periodic do not minimize beyond their $\theta$-period.






\subsection*{Organization of the paper}

Section \ref{Sec:se2-sR} introduces the group $\SE2$ as a \sR manifold. It also describes the metabelian structure. Sub-Section \ref{sec:cot-bund} presents the Hamiltonian function governing the \sR geodesic flow. It explains symplectic reduction by the normal subgroup $\R^2$ and writes down the reduced Hamiltonian $H_{\mu}$. It states and proves the \textbf{Background Theorem}. It describes some symmetries of the \sR geodesic flow and classifies the \sR geodesics. Sub-Section \ref{subsec:proof-periodic} proves Theorem \ref{main:the-periodic}.

Section \ref{sec:metri-line} study the metric lines in $\SE2$. Sub-Section \ref{sec:cal-fuc} introduces the Hamiltonian-Jacobi equation and its relation to the eikonal equation. It formally defines a calibration function and solves the \sR eikonal equation on $\SE2$. Sub-Section \ref{subsec:min-fuc} presents our method to prove that the type line and heteroclinic geodesics are metric lines. Sub-Section \ref{sec:cut-time} presents the definition of a cut-time and verifies that an upper bound for the cut-time for the $\theta$-periodic geodesics is its $\theta$-period. Sub-Section \ref{Sec:proof-the-main-2} summarizes our results and proves Theorem \ref{main:the-1}. 

The Appendix \ref{ape:mane} shows the relation between the calibration function used to prove Theorem \ref{main:the-1} and Ma\~{n}e's critical value, a fundamental concept from weak KAM theory.

\subsection*{Acknowledgements}
This research was conducted during the summer of 2024 as part of the University of Michigan, Ann Arbor math REU program. We thank Prof. Tasho Kaletha, Annie Winkler, and the faculty at the University of Michigan for organizing this REU. The author, S. Ku, would like to thank Prof. Asaf Cohen for providing funding under the National Science Foundation Grant number DMS-2006305.

\section{The Euclidean group as a \sR manifold}\label{Sec:se2-sR}
\subsection{The Euclidean group as a \sR manifold}

The Special Euclidean group is a $3$-dimensional Lie group. If $(\theta,x,y)$ in $(0,2\pi) \times \R^2$ are local coordinates, then a point $g$ in $\SE2$  has a matrix representation given by
$$ g = \begin{pmatrix}
    cos \theta & -\sin \theta & x \\
    \sin \theta & \cos \theta & y \\
    0 & 0 & 1 \\
\end{pmatrix},$$

The Lie algebra $\sea$ is given by
$$ E_{\theta} = \begin{pmatrix}
    0 & -1 & 0 \\
    1 & 0  & 0 \\
    0 & 0 & 0 \\
\end{pmatrix},\;\; E_{u} = \begin{pmatrix}
    0 & 0 & 1 \\
    0 &  0 & 0 \\
    0 & 0 & 0 \\
\end{pmatrix},\;\text{and}\;E_{v} = \begin{pmatrix}
    0 & 0 & 0 \\
    0 & 0  & 1 \\
    0 & 0 & 0 \\
\end{pmatrix},$$
then the Lie bracket relations are
\begin{equation}\label{eq:Lie-brac-rel}
[E_{\theta},E_u] = E_v,\;\; [E_{\theta},E_v] = -E_u,\;\;\text{and}\;\; [E_u,E_v] = 0.    
\end{equation}

Let us formalize the definition of a metabelian group. 
\begin{defi}
    We say a group $\mathbb{G}$ is metabelian if the commutator group $[\mathbb{G},\mathbb{G}]$ is abelian.
\end{defi}

In the case of a connected  Lie group, being metabelian is equivalent to $[\mathfrak{g},\mathfrak{g}]$ being abelian. Equation \eqref{eq:Lie-brac-rel} implies that the sub-algebra $[\sea,\sea]$ is an abelian ideal, since it is spanned by $\{E_u,E_v\}$. It follows that $\SE2$ is a metabelian Lie group. 

For the rest of the paper, we will write the left-invariant vector fields in terms of the operators $\frac{\partial}{\partial \theta}$, $\frac{\partial}{\partial x}$ and $\frac{\partial}{\partial y}$, rather than use matrix representation. Thus, the left-invariant vector fields are given by
\begin{equation} \label{eq:left-inv}
X_{\theta} = \frac{\partial}{\partial \theta}, \;\; X_u = \cos\theta\frac{\partial}{\partial x} + \sin \theta \frac{\partial}{\partial y}\;\; \text{and} \;\; X_v = -\sin \theta \frac{\partial}{\partial x} + \cos \theta\frac{\partial}{\partial y}.
\end{equation}

The frame $\{ X_\theta, X_u \}$ defines a non-integrable distribution $\D$. We declare the frame orthonormal to equip $\D$ with a \sR inner product. We say that a smooth curve $\gamma(t)$ in $\SE2$ is horizontal if $\dot{\gamma}(t)$ is tangent to $\D$ for all time $t$.

\subsection{\SR geodesic flow}\label{sec:cot-bund}

Let us consider the cotangent bundle  $T^*\SE2$ with the canonical coordinates $(p,g) = (p_{\theta},p_x,p_y,\theta,x,y)$. The left-invariant momentum functions associated to the left-invariant vector fields are given by
\begin{equation}
P_{\theta} = p_{\theta}, \;\; P_u = p_x\cos\theta + p_y \sin \theta \;\; \text{and} \;\; P_v = -p_x\sin \theta + p_y \cos \theta.
\end{equation}
For the definition of momentum functions, consult \cite[sub-Chpater 3.4]{agrachev} or \cite[Deﬁnition 1.5.4]{tour}.

Therefore, the Hamiltonian function governing the \sR geodesic flow is 
\begin{equation}
  \label{eq:ham-function}
  H_{sR}(p,g) = \frac{1}{2} ( P_{\theta}^2 + P_u^2 ) = \frac{1}{2} ( p_{\theta}^2 + ( p_x \cos \theta + p_y \sin \theta )^2 ) .
\end{equation}
For the formal definition of the \sR geodesic flow, refer to \cite[sub-sub-Chapter 4.7.2]{agrachev} or \cite[18, sub-Chapter 1.4]{tour} (an alternative name for the Hamiltonian function is \sR kinetic energy).

We notice that if $(p(t),\gamma(t)) \in T^* \SE 2$ is a solution to the Hamiltonian system defined by equation \eqref{eq:ham-function}, then $\gamma(t)$ is a horizontal curve. Indeed, Hamilton equations and the fact that the momentum functions  $P_\theta$ and $ P_{u}$ are linear in $p_\theta$, $p_x$ and $p_y$ imply
\begin{equation}\label{eq:dif-eq-base}
    \dot{\gamma}(t) = P_\theta X_{\theta}(\gamma(t)) + P_{u} X_{u}(\gamma(t)).
\end{equation}
When we choose the energy level $H(p,g) = \frac{1}{2}$, the corresponding geodesic $\gamma(t)$ is parameterized by arc length. 


Every cotangent bundle $T^*M$, where $M$ is a smooth manifold, possesses a natural Poisson bracket $\{\cdot,\cdot\}$. A Poisson bracket $\{\cdot,\cdot\}$ is a bi-linear operator on the space of smooth functions $C^{\infty}(M)$, with the property that $(C^{\infty}(M),\{\cdot,\cdot\})$ is a Lie algebra, and $\{\cdot,\cdot\}$ is a derivation, i.e., for all $F, G$ and $H$ in $C^{\infty}(M)$ the following expression holds
$$\{FG,H\} = F\{G,H\} + G\{F,H\}.$$
An alternative way to write the evolution of a function $F$ through the Hamiltonian flow of $H$ is given by $\dot{F} = \{F,H\}.$

If $P_X$ and $P_Y$ are the momentum functions associated with the vector fields $X$ and $Y$, then the relation between the Poisson bracket and vector field bracket is given by
 $$\{ P_X,P_Y\} = -P_{[X,Y]}.$$
To compute the differential equations for $P_{\theta}$, $P_{u}$ and $P_v$, we use the above relation to find
\begin{equation}\label{eq:dif-eq-fiber}
    \dot{P}_{\theta} = -P_{u}P_{v},\;\; \dot{P}_{u} =  P_{\theta}P_{v}, \;\;\text{and}\;\; \dot{P}_{v} = -P_{\theta}P_{u}.
\end{equation}

Since the Hamiltonian function $H (p,g)$ does not depend on the variables $x$ and $y$,  $p_x$ and $p_y$ are constant motions. In other words, the Hamiltonian function $H (p,g)$ is invariant under the Hamiltonian action of $\R^2$. Therefore, this action defined a momentum map given by 
$$ J(p,g) = (p_x,p_y) = \mu \in \mathbb{R}^2,$$
where we identify again $(\mathbb{R}^2)^*$ with $\R^2$ itself. For the formal definition of the momentum map, refer to \cite{ortega-ratiu-momentum-maps} or \cite[Appendix 5]{Arnold}. If $(p(t),\gamma(t))$ is a solution of the \sR geodesic flow, then we say that a geodesic $\gamma(t)$ has \textbf{momentum} $\mu$ if $J(p(t),\gamma(t)) = \mu$.

\subsubsection{Reduced dynamics}\label{sub-sec:red-dyn}

The semidirect product structure of $\SE2$ implies that the cotangent bundle $T^*\SE2$ is symplectic diffeomorphic to $T^*\SO2 \times T^* \R^2$, where we identify $T^*\R^2$ with $\R^2 \times \R^2$.   
Let us consider the level set $\mu = (a,b)$, then the inverse image $J^{-1}(\mu)$ is diffeomorphic to $T^*\SO2 \times \R^2 \times \mu$. Being $\R^2$ an abelian group, its isotropic group $\R^2_{\mu}$ is equal to $\R^2$. Thus the reduced space $T^*\SE2 //\R^2_{\mu} := J^{-1}(\mu) / \R^2_{\mu}$ is symplectic diffeomorphic to $T^*\SO2$. If $\pi_{\mu}: T^*\SE2 \to T^*\SO2$ is the symplectic map provided by the symplectic reduction, then the reduced Hamiltonian $H_{\mu}: T^*\SO2 \to \R$ is the unique function such that $H_{sR}|_{J^{-1}(\mu)} = H_{\mu} \circ \pi_{\mu}$. Therefore, the reduced Hamiltonian $H_{\mu}$ is given by 
\begin{equation}\label{eq:red-ham}
    H_{\mu}(p_{\theta},\theta) = \frac{1}{2}(p_{\theta}^2 + R^2 \cos^2 (\theta -\delta ) ), 
\end{equation}
where the bijection between $(a,b)$ and $(R,\delta)$ is given by $(R\cos\delta,R\sin\delta) = (a,b)$. The reduced Hamilton equations are 
\begin{equation}\label{eq:red-ham-eq}
  \dot{p}_\theta = R^2 \cos (\theta -\delta ) \sin(\theta -\delta )  \quad \text{and} \quad \dot{\theta} = p_\theta . 
\end{equation}

We notice that component $\theta$ of the reduced dynamic lies on a closed interval called the Hill interval. Let us introduce its formal definition. 

\begin{defi}\label{def:hill-interval}
    Let $H_{\mu}$ be a reduced Hamiltonian for the parameter $\mu$. We say that $I_{\mu}$ is a hill interval of $H_{\mu}$, if $I_{\mu}$ is a closed interval such that $R^2 \cos^2 (\theta -\delta ) < 1$ for all $\theta$ in the interior of $I_{\mu}$, and $R^2 \cos^2 (\theta -\delta ) = 1$ for $\theta$ in the boundary of $I_{\mu}$. We say that $Hill(\mu)$ is the Hill region of $H_{\mu}$ if $Hill(\mu)$ is the union of the Hill intervals for $H_{\mu}$.
\end{defi}

We think of a point $(p_\theta,\theta)$ in $T^*\SO2$ as a point in the cylinder $\R \times \SO2 \simeq \R \times \mathbf{S}^1$. The level set $H_{\mu}^{-1}( \frac{1}{2})$ is the curve $\alpha_{\mu}$ given by the equation \eqref{eq:alp-cur}.


\subsubsection{Background Theorem} \label{subsub:back-ground}

This sub-Section presents the method for building \sR geodesics and proving the \textbf{Background Theorem}. Let us provide the prescription: consider the initial value problem given by the Hamilton equations \eqref{eq:red-ham-eq} and the initial conditions $\alpha(t_0)$ in $\alpha_{\mu}$. Having found the solution $(p_\theta(t),\theta(t))$, we define a curve $\gamma(t)$ by the differential equation 
\begin{equation}
 \label{eq:hor-lift}   \dot{\gamma}(t) = \dot{\theta}(t) X_{\theta}(\gamma(t)) + R\cos(\theta(t)-\delta) X_{u}(\gamma(t)),
\end{equation}
where we used the reduced Hamilton system, given by equation \eqref{eq:red-ham-eq}, to identify $\dot{\theta}(t)$ with $p_\theta(t)$. The curve $\gamma(t)$ is defined for all time $t$ by the completeness of $H_{\mu}$. The  \textbf{Background Theorem} states that $\gamma(t)$ is a geodesic with momentum $\mu$.

\begin{backgroundtheorem}\label{the:back-grou}
  The above prescription yields a \sR geodesic in $\SE2$ with momentum $\mu$ parameterized by arc length. Conversely, every geodesic in $\SE2$ parameterized by arc length with momentum $\mu$ can be achieved by this prescription applied to the curve $\alpha_{\mu}$.
\end{backgroundtheorem}

\begin{proof} 
Let $\gamma(t)$ be a curve in $\SE2$ defined by equation \eqref{eq:hor-lift} for a fixed value of $\mu$. By construction, the curve $\gamma(t)$ is tangent to the distribution, and by comparing equations \eqref{eq:hor-lift} and \eqref{eq:dif-eq-base}, we conclude that it is enough to prove that the restriction of the left-invariant momentum functions $P_{\theta}(t) $, $P_{u}(t)$, and $P_{v}(t)$ restricted to the level set $J^{-1}(\mu)$ are equal to the functions 
\begin{equation*}
    F_{\theta}(t) = \dot{\theta}(t), \; \; F_{u}(t) = R\cos(\theta(t)-\delta), \;\; \text{and} \; \;  F_{v}(t) = -R\sin(\theta(t)-\delta),
\end{equation*}
respectively. Thus, we must prove that $F_{\theta}(t)$, $F_{u}(t)$, and $F_{\mu}(t)$  satisfy the equations given by \eqref{eq:dif-eq-fiber}. Using the reduced Hamilton equations given by equation \eqref{eq:red-ham-eq}, we have
\begin{equation}\label{eq:proof-back-gro}
    \begin{split}
        \dot{F}_{\theta}(t) & = R^2 \cos(\theta(t)-\delta) \sin(\theta(t) - \delta) =  -F_u(t) F_v(t), \\
        \dot{F}_{u}(t) & = - R  \sin(\theta(t) - \delta) \dot{\theta} =  F_\theta(t)   F_v(t), \\
        \dot{F}_{v}(t) & =  -R  \cos(\theta(t) - \delta) \dot{\theta} = -F_\theta(t) F_u(t). \\
    \end{split}
\end{equation}
Therefore, the equations in \eqref{eq:proof-back-gro} are identical to those from equation \eqref{eq:dif-eq-fiber}. We conclude that $\gamma(t)$ is a \sR geodesic in $\SE2$ with momentum $\mu$ by construction.
    
Conversely, let $\gamma(t)$ be an arbitrary geodesic in $\SE2$ parameterized by arc length with momentum $\mu$. The restriction of the Hamiltonian $H$ to the level set $J^{-1}(\mu)$ is, by definition, the reduced Hamiltonian $H_{\mu}$; the coordinates $p_\theta$ and $\theta$ satisfy the reduced Hamiltonian equations \eqref{eq:red-ham-eq}.  In addition, the momentum functions $P_{\theta}(t) $ and $P_{u}(t)$ restricted to the level set $J^{-1}(\mu)$ have the form  $P_{\theta}(t) = \dot{\theta}$ and $P_{u}(t) = R \cos(\theta(t) - \delta)$, and the equation \eqref{eq:dif-eq-base} is identical to \eqref{eq:hor-lift}. Thus, $\gamma(t)$ is achieved by the prescription applied to the curve $\alpha_{\mu}$. 
\end{proof} 

The following proposition describes a property of curves in $\R^2$ which are \sR geodesics projected by the canonical projection $\pi_{\R^2}: \SE2 \to \R^2 \simeq \SE2 / \SO2$. 
\begin{lemma}\label{lem:curv}
    Let $\pi_{\R^2}:\SE2 \to \R^2 $ be the canonical projection, which in coordinates is given by $\pi_{\R^2}(\theta,x,y):= (x,y)$. Let $\gamma(t)$ be a geodesic in $\SE2$, then the curvature of $\pi_{\R^2}(\gamma(t))$ is given by $\dot{\theta}$. 
\end{lemma}

\begin{proof}
    Let $\gamma(t)$ be a \sR geodesic in $\SE2$. The Background Theorem \ref{the:back-grou} implies that if $c(t) := \pi_{\R^2}(\gamma(t))$, then  $\dot{c}(t) = P_u(\cos(\theta(t),\sin(\theta(t))$. So the unitary tangent vector is $(\cos(\theta(t),\sin(\theta(t))$, and by differential geometry $\dot{\theta}$ is the curvature of $c(t)$. 
\end{proof}

\subsubsection{Symmetries of the reduced Hamiltonian flow}

Let us describe the curve $\alpha_{\mu}$ to understand the symmetries of the geodesic flow. The following proposition classifies the level set $H_{\mu}^{-1} = \alpha_{\mu}$.
\begin{lemma}\label{lem:alp-clas}
For different values of $R$, the level set $\alpha_{\mu}$  is described as follows: 
\begin{itemize}
    \item If $R> 1,\alpha_{\mu}$ consists precisely of two contractible, simple, and closed smooth curves. 
    
    The first curve is given by
    $$ \alpha_{\mu}^{1} = \{ (p_{\theta},\theta) = (\pm \sqrt{1-R^2\cos^2(\theta-\delta)}, \theta)  \text{ } | \text{ }  \theta \in  I_{\mu}^1 \}, $$
    here $I_{\mu}^1$ is the Hill interval $[\theta_{min}^1,\theta_{max}^1]$, where $\theta_{min}^1$  and $\theta_{max}^1$ are given by the solutions to the equation $\theta = \arccos(\frac{1}{R}) + \delta$ using the principal branch of $\arccos$. 

    The second curve is given by
        $$ \alpha_{\mu}^{2} = \{ (p_{\theta},\theta) = (\pm \sqrt{1-R^2\cos^2(\theta-\delta)}, \theta)  \text{ } | \text{ } \theta \in I_{\mu}^2   \}, $$
    here $I_{\mu}^2$ is the Hill interval $[\theta_{min}^2,\theta_{max}^2]$, where $\theta_{min}^2$ and $\theta_{max}^2$ are given by the solution to the equation $\theta = \arccos(\frac{1}{R}) + \delta + \pi$ using the principal branch. 
    \item If $R=1$, then $\alpha_{\mu}$ consists of precisely one non-contractible, non-simple, and closed curve. 
    \item If $0 < R <1$, then $\alpha_{\mu}$ consists of precisely two non-contractible, simple, and closed smooth curves given by
    $$ \alpha_{\mu}^{\pm} = \{ (p_{\theta},\theta) = (\pm \sqrt{1-R^2\cos^2(\theta-\delta)}, \theta)  \text{ } | \text{ } \theta \in \mathbf{S}^1  \} $$
\end{itemize}
\end{lemma}

\begin{proof}
Let us discuss first the smoothness of the level set $\alpha_{\mu}$. The regular value theorem says that $\alpha_{\mu}$ is smooth if 
$$dH_{\mu}\big|_{\alpha_{\mu}} = (p_{\theta}, -R^2\cos(\theta - \delta)\sin (\theta - \delta))  \neq 0. $$
If $p_{\theta} \neq 0$ then $dH_{\mu}\big|_{\alpha_{\mu}} \neq 0$. Thus, focusing on the case $p_{\theta} = 0$ is sufficient. The condition $p_{\theta} = 0$ implies $R^2\cos^2(\theta - \delta) = 1$, thus $dH_{\mu}\big|_{\alpha_{\mu}} = 0$ if and only if $(p_{\theta},\theta) = (0, \delta)$ or $(p_{\theta},\theta) = (0, \pi + \delta)$. The conditions $\theta = \delta$ or $\theta = \pi + \delta$ imply $R = 1$. Therefore, $\alpha_{\mu}$ is a smooth curve if $R \neq 1$.

If $R > 1$, then the level set $\alpha_{\mu}$ is well defined when $0 \leq 1 -R^2 \cos^2( \theta - \delta)$. In that case, $\theta$ has two disjoint intervals where this inequality holds: the Hill intervals $I_{\mu}^1$ and $I_{\mu}^2$. Let us parameterize the curves by $\theta$, then when $p_{\theta} = 0$, the positive and negative touch, making the level set consists of two simple closed and contractible curves. 

If $R = 1$, the level set $\alpha_{\mu}$ is well defined for all $\theta$. In addition, we can parameterize the level set by the expression $p_{\theta}= \pm |\sin(\theta-\delta)|$.  If $p_{\theta} = 0$, the parametrizations with a positive and negative sign coincide, making the level set a non-simple, noncontractible, and closed curve. 

If $R>1$, the level set $\alpha_{\mu}$ is well defined for all $\theta$. If we parameterize the curves by $\theta$, then $p_{\theta} \neq 0$. Thus, the curves  $\alpha_{\mu}^{+}$ and $\alpha_{\mu}^{-}$ do not touch, making the level set consists of two non-contractible, simple, and closed curves.
 \end{proof}

\begin{Remark}\label{rem:alp-ext} We extend the definition of the curves $\alpha_{\mu}^1$ and $\alpha_{\mu}^2$  to the case $R=1$ where $I_{\mu}^1 = [\delta,\pi+\delta]$, and $I_{\mu}^2 = [\pi+\delta,2\pi+\delta]$ are the domains of the curves, respectively. Similarly, we extend the definition of the curves $\alpha_{\mu}^+$ and $\alpha_{\mu}^-$ to the case $R=1$. \end{Remark}

See Figure \ref{fig:three graphs} for a better understanding of  Lemma \ref{lem:alp-clas} and Remark \ref{rem:alp-ext}.

The following lemma describes the symmetries of the reduced Hamiltonian flow, which helps us study the symmetries of the geodesic flow.
\begin{lemma}\label{lem:sym-red-ham}
The reduced Hamiltonian has the following symmetries:
\begin{itemize}
    \item If $R \geq 1$ and $(p_{\theta}(t),\theta(t))$ is a solution laying in $\alpha_{\mu}^1$, then $(p_{\theta}(t),\theta(t) + \pi)$ is a solution laying in $\alpha_{\mu}^2$. 
    \item If $R \leq 1$ and $(p_{\theta}(t),\theta(t))$ is a solution laying in $\alpha_{\mu}^+$, then $(-p_{\theta}(-t),\theta(-t))$ is a solution laying in $\alpha_{\mu}^{-}$.
\end{itemize}

\end{lemma}

\begin{proof}
   If $R \geq 1$, we notice that if $(p_\theta,\theta)$ is a point in $\alpha_{\mu}^1$, then $(p_\theta,\theta + \pi)$ is a point in $\alpha_{\mu}^2$. Moreover, if   $(p_\theta(t),\theta(t))$ is a solution, then   $(\tilde{p}_{\theta},\tilde{\theta}) = (p_\theta(t),\theta(t) + \pi)$ is also a solution. Indeed, $(\tilde{p}_{\theta}(t),\tilde{\theta}(t))$  satisfies the differential equations
   \begin{equation*}
   \begin{split}
     \dot{\tilde{\theta}}     & = \dot{\theta} = p_{\theta} = \tilde{p}_{\theta}, \\
     \dot{\tilde{p}}_{\theta} & = \dot{p}_{\theta} = R^2 \cos(\theta - \delta)\sin(\theta-\delta) \\
        &= R^2 \cos(\tilde{\theta}-\delta - \pi)\sin(\tilde{\theta}-\delta - \pi) = R^2 \cos(\tilde{\theta}-\delta )\sin(\tilde{\theta}-\delta).
   \end{split}
   \end{equation*}
   
If $R \leq 1$, we notice that if $(p_\theta,\theta)$ is a point in $\alpha_{\mu}^+$, then $(-p_\theta,\theta)$ is a point in $\alpha_{\mu}^{-}$. Moreover, if   $(p_\theta(t),\theta(t))$ is a solution, then   $(\tilde{p}_{\theta},\tilde{\theta}) = (-p_\theta(-t),\theta(-t))$ is also a solution. Indeed, $(\tilde{p}_{\theta}(t),\tilde{\theta}(t))$  satisfies the differential equations
   \begin{equation*}
   \begin{split}
     \dot{\tilde{\theta}}(t)     & = -\dot{\theta}(-t) = -p_{\theta}(-t) = \tilde{p}_{\theta}(t), \\
     \dot{\tilde{p}}_{\theta}(t) & = \dot{p}_{\theta}(-t) = R^2 \cos(\theta(-t) - \delta)\sin(\theta(-t)-\delta) = R^2 \cos(\tilde{\theta}(t) - \delta)\sin(\tilde{\theta}(t) -\delta).
   \end{split}
   \end{equation*}
We call such a property time reversibility.
\end{proof}

\begin{figure}
     \centering
     \begin{subfigure}[b]{0.3\textwidth}
         \centering
         \includegraphics[width=1\textwidth]{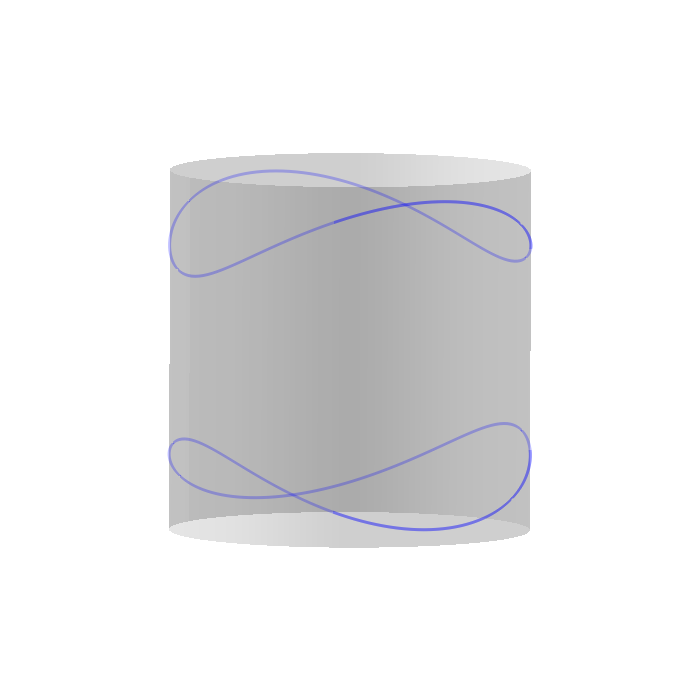}
         \caption{$R>1$}
         \label{fig:Rless1}
     \end{subfigure}
     \hfill
     \begin{subfigure}[b]{0.3\textwidth}
         \centering
         \includegraphics[width=1\textwidth]{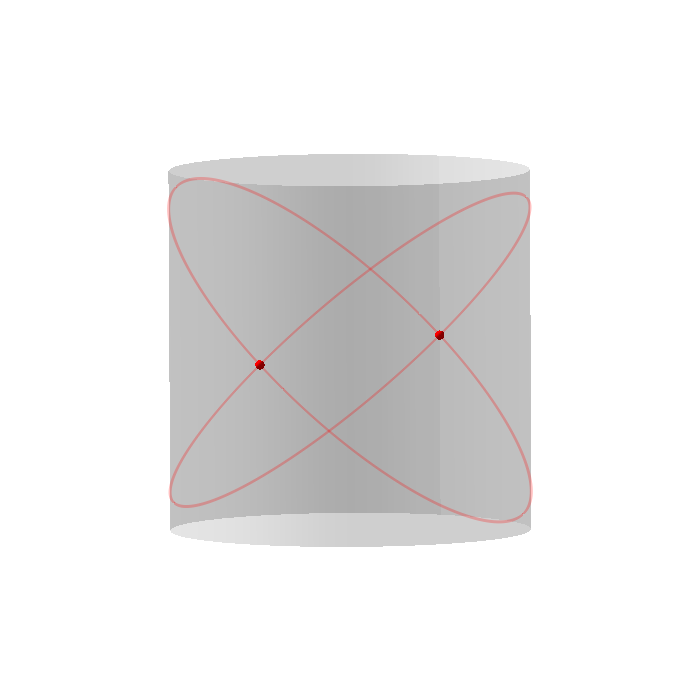}
         \caption{$R=1$}
         \label{fig:Requal1}
     \end{subfigure}
     \hfill
     \begin{subfigure}[b]{0.3\textwidth}
         \centering
         \includegraphics[width=1\textwidth]{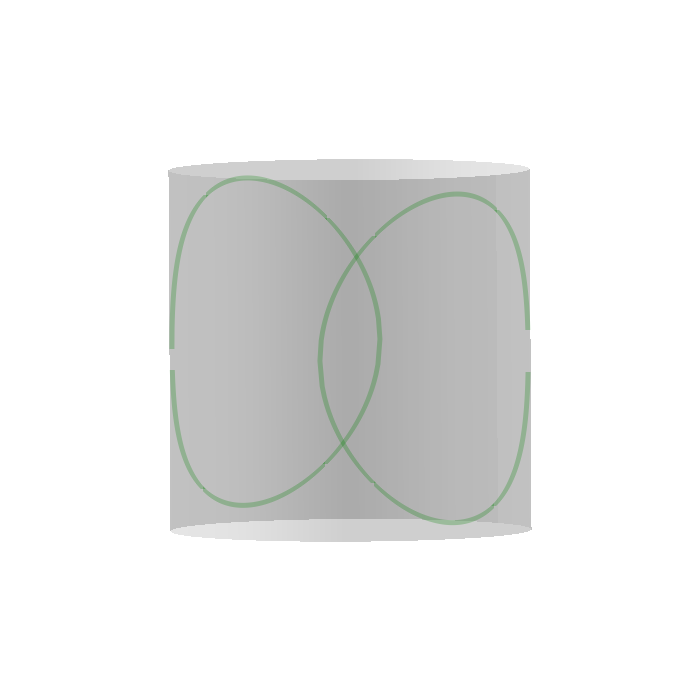}
         \caption{$R<1$}
         \label{fig:Rbigger1}
     \end{subfigure}
        \caption{The panels show the curve $\alpha_{\mu}$ for the three different cases}
        \label{fig:three graphs}
\end{figure}

\subsubsection{Classification of \sR geodesics}\label{subsubsec:clas-geo}

In this sub-sub-Section, we classify the \sR geodesics according to their reduced dynamics. If $\gamma(t)$ is a geodesic parameterized by arc-length in $\SE2$, then $\gamma(t)$ is one of the following three types:

(\textbf{Line}) We say a geodesic $\gamma(t)$ is of the type line if  $\dot{\theta} = 0$. A geodesic is of the type line if its reduced dynamic is trivial, i.e., if $R= 1$ and $\theta = \delta$ or $\theta = \delta + \pi$. 

(\textbf{Heteroclinic}) We say a geodesic $\gamma(t)$ is of the type heteroclinic if its reduced dynamic is heteroclinic. The reduced dynamic is heteroclinic if and only if $R=1$ and $\dot{\theta} \neq 0$. 

(\textbf{$\theta$-periodic}) We say a geodesic $\gamma(t)$ is of the type $\theta$-periodic if its reduced dynamic is periodic. The reduced dynamic is periodic if and only if $R \neq 1$. \smallskip

Lemma \ref{lem:curv} states that a geodesic projected to the plane $\R^2$ has curvature $P_{\theta} = p_{\theta}$, then following the terminology from \cite{Moiseev2010}, we also distinguish whether the projection of the geodesics has inflection points.

(\textbf{Inflection}) We say a geodesic $\gamma(t)$ is of the type inflection  if the projection $\pi_{\R^2}(\gamma(t))$ has inflection points. $\gamma(t)$  has inflection points if and only if $R > 1$.

(\textbf{Non-Inflection}) We say a geodesic $\gamma(t)$ is of the type non-inflection if the projection $\pi_{\R^2}(\gamma(t))$ does not have inflection points. $\gamma(t)$ is non-inflection if and only if $R \leq 1$. \smallskip

\begin{figure}
     \centering
     \begin{subfigure}[b]{0.3\textwidth}
         \centering
         \includegraphics[width=1\textwidth]{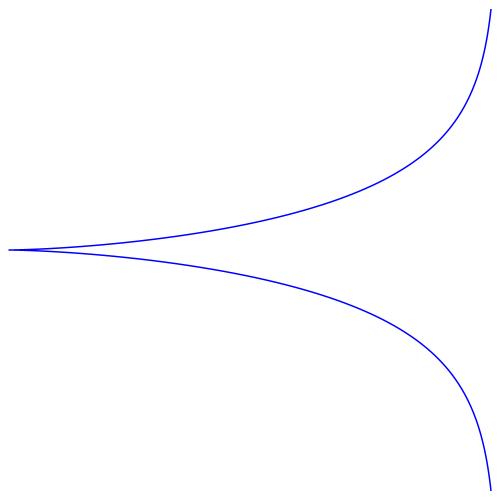}
         \caption{Heteroclinic}
     \end{subfigure}
     \hfill
     \begin{subfigure}[b]{0.3\textwidth}
         \centering
         \includegraphics[width=1\textwidth]{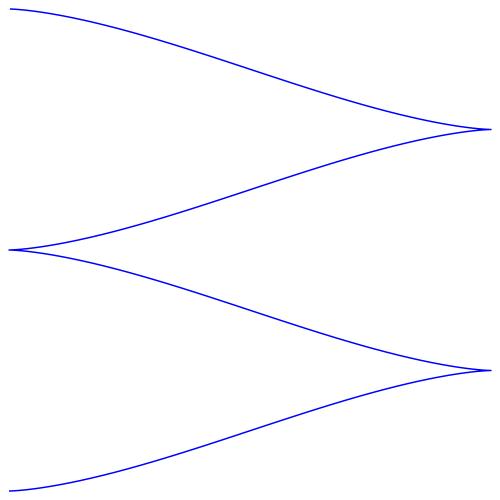}
         \caption{Inflection}
     \end{subfigure}
     \hfill
     \begin{subfigure}[b]{0.3\textwidth}
         \centering
         \includegraphics[width=1\textwidth]{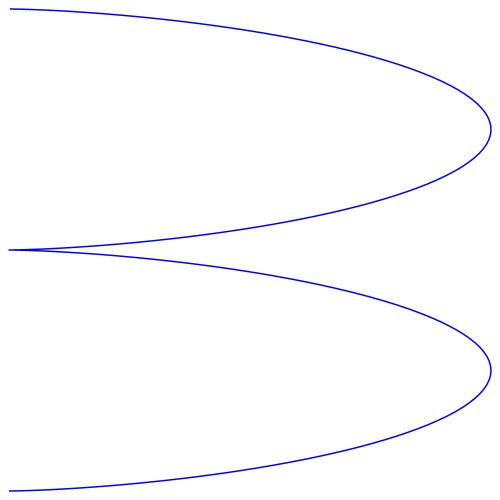}
         \caption{Non-Infelction}
     \end{subfigure}
        \caption{The panels show the projection to the plane $\R^2$ for each type of geodesic}
        \label{fig:geo}
\end{figure}

 The following proposition gives explicit formulas for the $\theta$-period and the change performed by the $x$ and $y$ coordinates during the period. 

\begin{Prop}\label{prp:period-map}
Let $\gamma(t)$ be a \sR geodesic of the type $\theta$-periodic with momentum $\mu$ and Hill interval $I_{\mu}$. Then, the $\theta$-period is given by 
\begin{equation*}
L(\mu) := \int_{I_{\mu}} \frac{d\theta}{\sqrt{1-R^2\cos^2(\theta-\delta)}}  .  
\end{equation*}
The changes $\Delta x(\mu)$ and $\Delta y(\mu)$ performed by the coordinates $x$ and $y$ after the geodesic travels a period $L(\mu)$ are given by
 \begin{equation*}
\Delta x(\mu) := \int_{I_{\mu}} \frac{ R\cos(\theta-\delta) \cos(\theta) d\theta}{\sqrt{1-R^2\cos^2(\theta-\delta)}}\;\; \text{and} \; \; \Delta y(\mu) := \int_{I_{\mu}} \frac{ R\cos(\theta-\delta) \sin(\theta) d\theta}{\sqrt{1-R^2\cos^2(\theta-\delta)}}.  
 \end{equation*}
 In addition, the changes $\Delta x(\mu)$ and $\Delta y(\mu)$ are independent of the initial point.
\end{Prop}

\begin{proof} 
Being the reduced system $H_{\mu}$ one degree of freedom, we reduce it to quadrature in the following way: Using the energy level $H_{\mu} = \frac{1}{2}$ and the second reduced Hamilton equation from \eqref{eq:red-ham-eq}, we have 
\begin{equation}\label{eq:theta-dot-sol}
    \dot{\theta} = \pm \sqrt{1-R^2\cos^2(\theta-\delta)}.
\end{equation}
We solve the differential equation using the separation of variables method. To compute $\Delta x(\mu)$ and $\Delta y(\mu)$, we integrate the coordinates $x$ and $y$ in the same way using equation \eqref{eq:hor-lift}. 

We remark that there is no ambiguity regarding the sign of $\dot{\theta}$. If $\dot{\theta}$ is positive in the interval $(t-\epsilon,t+\epsilon)$ for some $\epsilon>0$, then the interval of integration $[\theta(t-\epsilon),\theta(t+\epsilon)]$ is positively oriented. Conversely, if $\dot{\theta}$ is negative in the interval $(t-\epsilon,t+\epsilon)$ for some $\epsilon>0$, then the interval o integration $[\theta(t-\epsilon),\theta(t+\epsilon)]$ is negatively oriented. Therefore, if $\dot{\theta}$ is negative, we utilize the positive root and integrate on the positively oriented interval $[\theta(t+\epsilon),\theta(t-\epsilon)]$. We make the convention to choose the positive root and integrate it using positively oriented intervals. For more details about this integration, refer \cite[Section 11]{Landau} for a general mechanical system. 
\end{proof}

\subsection{Proof of Theorem \ref{main:the-periodic}}\label{subsec:proof-periodic}

\begin{proof}
     When $R=0$. We have that $x(t) = x_0 $ and $ y(t) = y_0$ are constant for all $t$ in $\R$ and $\theta(t)$ is periodic, so the geodesic is periodic.

     Let us assume that $\gamma(t)$ is a periodic \sR geodesic, then Proposition \ref{prp:period-map} implies that a geodesic $\gamma(t)$ of the type $\theta$-periodic is periodic if and only if $\Delta x(\mu) =0$ and  $\Delta y(\mu) = 0$. Therefore, Proposition \ref{prp:period-map} and the cosine addition formula implies
 \begin{equation*}
     \begin{split}
        0 = a \Delta x(\mu) + b \Delta y(\mu) &= R^2 \int_{\beta_{\mu}} \frac{ \cos^2(\theta-\delta)  d\theta}{\sqrt{1-R^2\cos^2(\theta-\delta)}}.
     \end{split}
 \end{equation*}
Where we used again that the relation $(a,b) = R(\cos(\delta),\sin(\delta))$. Since the integral from the above equation is positive, we conclude $R = 0$.
\end{proof}

\section{Metric lines}\label{sec:metri-line}
\subsection{Hamilton-Jacobi Equation and Calibration Functions} \label{sec:cal-fuc}

We will utilize Hamilton-Jacobi theory to build a calibration function. Let us first introduce the proper definitions: Given a Riemannian manifold $M$ and a Hamiltonian function $H: T^*M \to \R$, the time-independent Hamilton-Jacobi equation is a partial differential equation in  $S: M \to \R$ given by
\begin{equation}
    H(dS,q)  = const, \label{eq: Hamilton Jacobi}
\end{equation} 
where $dS$ is the differential of $S(q)$.

When the Hamiltonian $H$ is purely kinetic, consult \cite[Chapter 2]{Arnold} for the formal definition of a kinetic energy, the Hamilton-Jacobi equation is also known as the \textbf{eikonal equation}, and we can rewrite equation \eqref{eq: Hamilton Jacobi} as $||\nabla S|| = 1$. If $S(q)$ is a solution to the eikonal equation, then $S(q)$ is the oriented distance from the point $q$ to a given sub-manifold; consult \cite{on-riem-man-grad} for more details about this property in the \Ri case.

When considering a \sR manifold with a \sR kinetic energy, the corresponding operator is the horizontal gradient, denoted by $\nabla_{hor} S$. Then the \sR eikonal equation is given by $||\nabla_{hor} S|| = 1$.  Let us introduce the definition of the horizontal gradient. 
\begin{defi}\label{def:hor-nab}
Let $M$ be a \sR manifold with distribution $\D$ and \sR inner product $\langle \cdot,\cdot \rangle$. Let $S:M\rightarrow \mathbb{R}$ be a $C^k(M)$ function with $k\geq 1$. We say that a vector field $\nabla_{hor}S$ is the horizontal gradient of $S(q)$, if $\nabla_{hor}S$ is the unique horizontal vector field satisfying
   \begin{equation*}
       \langle\nabla_{hor}S,v\rangle_q=dS_q(v),\; 
   \end{equation*}
    for all $v$ in $\D_q$ and  $q$ in $M$.
    \end{defi}

If $H_{sR}$ is the \sR kinetic energy given by equation \eqref{eq:ham-function}, then the \sR eikonal equation is the partial differential equation given by


\begin{equation}\label{eq:eikonal}
    1 = (\frac{\partial S}{\partial \theta})^2 + (\cos\theta\frac{\partial S}{\partial x} + \sin\theta\frac{\partial S}{\partial y})^2.
\end{equation}

Let us solve the \sR eikonal equation by reducing it to an ordinary differential equation. Consider the ansatz
\begin{equation}
    S_{\mu}(\theta,x,y) = f(\theta) + R\cos(\delta) x + R \sin (\delta) y. 
\end{equation}
Substituting into equation \eqref{eq:eikonal} we find that
\begin{align*}
    1 &= (f'(\theta))^2+ R^2\cos^2(\theta -\delta). 
\end{align*}
Therefore $f(\theta)$ is a solution to the differential equation 
\begin{equation*}
    f'(\theta)  = \pm \sqrt{ 1 - R^2 \cos^2 (\theta  - \delta)},
\end{equation*}
and the solution $S_{\mu}$ is given by
\begin{align}\label{eq:cal-fun}
    S_{\mu}^{\pm}(\theta,x,y) &= \pm \int \sqrt{1 - R^2 \cos^2(\theta - \delta)} d \theta + R  \cos( \delta )x + R  \sin (\delta) y .
\end{align}

\begin{Prop}\label{pro:hor-vect-field}
    Let $S_{\mu}^{\pm}$ be a solution to the eikonal equation given by equation  \eqref{eq:cal-fun}, and $\gamma(t)$ be a geodesic with momentum $\mu$, then $\nabla_{hor} S_{\mu}^{\pm} = \dot{\gamma}$ whenever $\dot{\theta}$ has the same sign as the root of \eqref{eq:cal-fun}. 
\end{Prop}

Before proving Proposition \ref{pro:hor-vect-field}, let us introduce the co-frame of left-invariant one-forms:
\begin{equation*}
    \Theta_\theta = d \theta ,\;\; \Theta_u = \cos \theta dx  + \sin \theta dy, \;\; \text{and} \;\; \Theta_v = -\sin \theta dx  + \cos \theta dy. 
\end{equation*}

\begin{Remark}
     Since $\Theta_{v}(X_{\theta}) = \Theta_{v}(X_{u}) = 0$, an alternative way to define the non-integrable distribution $\D$ over $\SE 2$ is as the kernel of $\Theta_{v}$. This tells us that $\SE2$ has the structure of a contact manifold. It is well known that a contact \sR structure does not have abnormal geodesics; consult \cite[Proposition 4.38]{agrachev} or \cite[sub-Chapter 5.2]{tour} for more details.
\end{Remark}

Let us prove Proposition \ref{pro:hor-vect-field}.
\begin{proof}
   
  Let $\gamma(t)$ be a geodesic of momentum $\mu$, without loss of generality, let us consider the positive root from equation \eqref{eq:cal-fun} and assume that $\dot{\theta}$ is positive; then we can write $dS^{+}_{\mu}$ in terms of the co-frame  
\begin{align*}
    dS_{\mu}^{+} &=    \sqrt{1 - R^2 \cos^2(\theta - \delta)} d \theta + R \cos (\theta - \delta)d  \Theta_u + R \sin (\theta - \delta)d  \Theta_v .
\end{align*}  
By Definition \ref{def:hor-nab}, the horizontal gradient  $\nabla_{hor} S_{\mu}^+$ is given by
$$\nabla_{hor} S_{\mu}^+ = \sqrt{1 - R^2 \cos^2(\theta - \delta)} X_{\theta} +  R \cos(\theta - \delta) X_u.$$
Comparing  equation \eqref{eq:hor-lift} with $\nabla_{hor} S_{\mu}^+$, we conclude that $\nabla_{hor} S_{\mu} = \dot{\gamma}(t)$, where we use $\dot{\theta} = \sqrt{1 - R^2 \cos^2(\theta - \delta)}$.
    
\end{proof}

Let us now introduce the definition of calibration functions.
\begin{defi} \label{def:cal-func}
   Let $M$ be a \sR manifold with distribution $\mathcal{D}$, we say that a function $S: M \to \R$ is a calibration function for the geodesic $\gamma(t)$ if the following conditions hold:
    \begin{itemize}
        \item $dS(\gamma'(t)) =1 $ for all $t$.
        
        \item $|dS(v)| \leq \|v\|_{sR}$ for all vectors $v$ in $\D$, where $\|\cdot\|_{sR}$ is the \sR norm given by $\|v\|_{sR} \coloneqq \sqrt{\braket{v,v}}$.  
    \end{itemize} 
\end{defi}

\begin{lemma}\label{lem:cal-func}
Let $\gamma(t)$ be a \sR geodesic with momentum $\mu$. If $S_{\mu}^{\pm}$ is the function given by equation \eqref{eq:cal-fun}, then $S_{\mu}^{\pm}$ is a calibration function for $\gamma(t)$ whenever $\dot{\theta}$ has the same sign that the root of equation \eqref{eq:cal-fun}. 
\end{lemma}

\begin{proof}
With loss of generality, let us consider the positive root from \eqref{eq:cal-fun} and $\dot{\theta} > 0$. Lemma \ref{pro:hor-vect-field} tells us that  $\dot\gamma(t)=\nabla_{hor}S_\mu^+$. Definition \ref{def:hor-nab} implies
\begin{align*}
    dS_{\mu}^+ (\dot\gamma(t) )&= \braket{\dot\gamma(t),\dot\gamma(t)} = 1,
\end{align*}
where we used the fact that $\gamma(t)$ is parametrized by arc length. We conclude that the first condition from Definition \ref{def:cal-func} is satisfied. To prove the second condition; let $v$  be an arbitrary vector in $\mathcal{D}$, then the Cauchy–Schwarz inequality implies that 
\begin{align*}
  |dS_{\mu}^+(v)| =|\langle \dot{\gamma}, v\rangle|
    \leq \|v\|_{sR}.
\end{align*}
Therefore, such $S^{\pm}_{\mu}$ is a calibration function for $\gamma(t)$ whenever $\dot{\theta}$ has the same sign as the root of \eqref{eq:cal-fun}.
\end{proof}

\subsubsection{Minimizing Method}\label{subsec:min-fuc}
The following proposition provides the conditions for a calibration function to ensure a geodesic is a metric line.
\begin{Prop}\label{prp:cal-func-region}
Let $M$ be a \sR Manifold. If $S: M \to \R$ is a $C^2$ global calibration function for the geodesic $\gamma(t)$, then $\gamma(t)$ is a metric line. 
\end{Prop}

\begin{proof}
    Let $S(q)$ be a $C^2$ global calibration function for a geodesic $\gamma(t)$ in $M$. The goal is to show that for every compact interval $[t_0,t_1]$, the curve $\gamma(t)$ is the shorter curve joining the points $\gamma(t_0) = q_0$ and $\gamma(t_1) = q_1$. 
    
    As $S(q)$ a $C^2$  globally defined function, it implies that $dS$ is a global exact 1-form. Let $\widetilde{\gamma}(t)$ be an arbitrary curve joining the points $q_0$ and $q_1$. Stoke's theorem implies
\begin{equation*}
      |\int_{\widetilde\gamma} dS| = |S(q_1)-S(q_0)| = |\int_{\gamma} dS | = |t_1-t_0| ,
\end{equation*}
where we use the first condition from Definition \ref{def:cal-func}. Furthermore, the second condition implies 
\begin{equation*}
   | \int_{\widetilde\gamma} dS|  \leq \int \|\dot{\widetilde\gamma}(t)\|_{sR} dt = \ell(\widetilde \gamma)
\end{equation*}
where $\ell(\widetilde \gamma))$ is the sub-Riemannian length of the curve $\widetilde \gamma(t)$. By the Cauchy-Schwarz equality for integrals, the equality holds if and only if $\dot{\widetilde\gamma}(t) =f \dot{\gamma}(t)$ for some scalar function $f$. In other words, $\tilde{\gamma}(t)$ is a reparameterization of the curve of $\gamma(t)$.
\end{proof}

Proposition \ref{prp:cal-func-region} suggests that to classify the metric lines on $\SE 2$, we need to study when the calibration function defined in equation \eqref{eq:cal-fun} is globally defined and $C^2$. Although the calibration function given by equation \eqref{eq:cal-fun} is not globally defined, the following proposition presents a new global calibration function. 

\begin{Prop} \label{prop:glo-cal-fun}
    The calibration function given by equation \eqref{eq:cal-fun} is not globally defined for every value of $R$. However, the following smooth function is globally defined 
    \begin{equation}\label{eq:glo-cal-fun}
        S_{\delta}^{\pm} (\theta,x,y) = \mp \cos(\theta - \delta)  + x \cos(\delta)  +  y \sin(\delta).
    \end{equation}
Moreover, $S_{\delta}^{\pm}$ is a calibration function for the heteroclinic and line geodesics with momentum $\mu$.
\end{Prop}
\begin{proof} 
     We notice that by construction, the differential one-form $dS_{\mu}^{\pm}$, given by equation \eqref{eq:cal-fun}, for a parameter $\mu$ is only defined in the closed region $\Omega_{\mu}:= Hill(\mu) \times \R^2$, where $Hill(\mu) $ is the Hill region from Definition \ref{def:hill-interval}. If $R>1$, then Lemma \ref{lem:alp-clas} implies that $\Omega_{\mu}$ is a proper subset of $\SE2$. So $S_{\mu}^{\pm}$ is not globally defined.   

     If $R \leq 1$, the differential one-form $dS_{\mu}$ is globally defined. However, the one-form $dS_{\mu}$ is not exact since the integral over a non-contractible loop is not zero. Indeed, if $\gamma(t)$ is a non-contractible loop in $\SE2$, then  
    \begin{equation*}
         \int_{\gamma} dS_{\mu}^{\pm} = \pm \int_0^{2\pi} \sqrt{1-R^2\cos^2(\theta-\delta)} d \theta \neq 0.
    \end{equation*}

    The function $S_{\delta}^{\pm}$, given by equation \eqref{eq:glo-cal-fun}, is globally defined and smooth by construction. Let us verify that $S_{\delta}^{\pm}$ is a calibration function: Let $\gamma(t)$ be a geodesic of the type line, then $\theta(t) = \delta$ for all time $t$ and the Background Theorem implies
     $$ \dot{\gamma}(t) = \cos(\delta) \frac{\partial}{\partial x} + \sin(\delta) \frac{\partial}{\partial y}.  $$
     Therefore $dS^{\pm}_{\delta}(\dot{\gamma}) = 1$. 

     Let $\gamma(t)$ be a \sR geodesic of the type heteroclinic. Without loss of generality, let us assume that $\gamma(t)$ is built by a solution to the reduced Hamiltonian system with initial condition in $\alpha_{\mu}^1 \cap \alpha_{\mu}^+$, i.e., $\theta(t)  \in (\delta,\delta+\pi)$ and $\dot{\theta}(t) > 0$. Then, the Background Theorem implies 
     \begin{equation*}
         \begin{split}
        \dot{\gamma}(t) & = \sqrt{1 -  \cos^2(\theta(t) - \delta)} X_{\theta} +  \cos(\theta(t) - \delta) X_u \\
                        & = \sin(\theta(t) - \delta) X_{\theta} + \cos(\theta(t) - \delta) X_u .
         \end{split}
     \end{equation*}
     We notice that $ 0 \leq \sin(\theta(t) - \delta)$, since $\theta(t)$ lays in the interval $[\delta,\delta+\pi]$.  Therefore $dS^{+}_{\delta}(\dot{\gamma}) = 1$.
     
     We conclude that $S_{\delta}^{+}$ satisfies the first condition from Definition \ref{def:cal-func} for geodesics of the type heteroclinic and line. The Cauchy–Schwarz inequality yields the second condition. Therefore, $S_{\delta}^{\pm}$ is a calibration function for geodesics of type heteroclinic and line with momentum $\mu$.    
\end{proof}

We are ready to prove that geodesics of the type heteroclinic and line are metric lines.
 \begin{Prop}  \label{Prop: Heteroclinic geod minimizing }
     Geodesics of the type line and heteroclinic metric lines. 
     \end{Prop}
\begin{proof}
Let $\gamma(t)$ be a \sR geodesic either of the type line or heteroclinic with momentum $\mu$. Propositions \ref{prp:cal-func-region} implies that the function $S_{\delta}^{\pm}$ is a smooth global calibration function for $\gamma(t)$. So, Proposition \ref{prop:glo-cal-fun} implies the desired result.
\end{proof}

\subsection{Cut times in the Special Euclidean group}\label{sec:cut-time}

We just completed the first part of Theorem \ref{main:the-1} by proving Proposition \ref{Prop: Heteroclinic geod minimizing }. To finish the proof, we will show that geodesics of the type $\theta$-periodic fail to qualify as metric lines. To do so, we will prove that their $\theta$-period is an upper bound for their cut time. Let us formalize the definition of the cut time.
\begin{defi}
    Let $\gamma(t)$ be a \sR geodesic parameterized by arc length. We define the cut time of $\gamma(t)$ as \begin{equation*}
        t_{cut}(\gamma) = \sup\{ t > 0 \; | \; \gamma|_{[0,t]} \; \text{is length minimizing}\}.
    \end{equation*}
\end{defi}

If two geodesics with the same initial point touch again after a positive time. In that case, they fail to minimize after this time. Refer to \cite[Lemma 5.2]{comparison-theorem} for the \Ri case. The following proposition demonstrates this is the case for the type $\theta$-periodic geodesics.

\begin{figure}
     \centering
     \begin{subfigure}[b]{0.3\textwidth}
         \centering
         \includegraphics[width=1\textwidth]{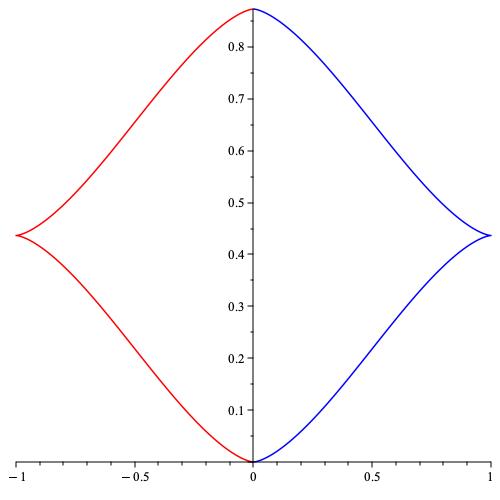}
         \caption{$R>1$}
     \end{subfigure}
     \hfill
     \begin{subfigure}[b]{0.3\textwidth}
         \centering
         \includegraphics[width=1\textwidth]{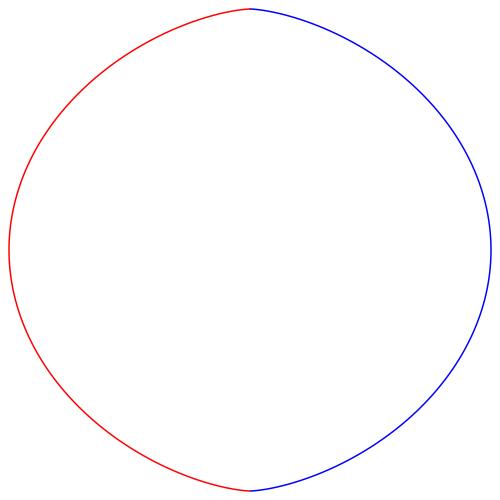}
         \caption{$R=1$}
     \end{subfigure}
        \caption{The panels show the cut points for each case}
        \label{fig:cut-points}
\end{figure}

\begin{Prop}\label{prop:cut-time}
If a geodesic is $\theta$-periodic with period $L(\mu)$, then $L(\mu)$ is an upper bound for the cut time. 
\end{Prop}

\begin{proof}
Let $\gamma(t)$ be an $L$ periodic sub-Riemannian geodesic with momentum $\mu$ and initial condition $\gamma(0) = (\theta_0,x_0,y_0)$, we will consider two cases when $\dot{\theta}_0 \neq 0$ and $\dot{\theta}_0 = 0$.  We remark that the second case only corresponds to geodesics of the type inflection $(R>1)$.

Case $\dot{\theta}_0 \neq 0$: there are exactly two geodesics with initial condition $\gamma(0)$ and momentum $\mu$, namely $\gamma(t)$ and $\tilde{\gamma(t)}$, the latter of which is defined as the one whose reduced dynamics is a solution to the Hamiltonian $H_{\mu}$ and have the initial condition $(\tilde{p}_{\theta}(0),\tilde{\theta}(0)) = (-\dot{\theta}_0,\theta_0)$. The time reversibility (see Lemma \ref{lem:sym-red-ham}) implies $\theta(t) = \tilde{\theta}(-t)$ for all $t$, and the periodicity gives $\theta(L) = \tilde{\theta}(L)$. 

Moreover, we claim that $\gamma(L) = \tilde{\gamma}(L)$. This is followed by Proposition \ref{prp:period-map}, which states that the difference in $\tilde{x}$ and $\tilde{y}$ over the period $L$ is the same as the difference for $\tilde{x}$ and $\tilde{y}$. That is to say 
\begin{align*}
\gamma(L) =  \gamma(0) + (0,\Delta x(\mu), \Delta y(\mu)) =\tilde{\gamma}(L)
\end{align*}
Therefore, we have constructed two distinct geodesics meeting at time $L$, implying that any geodesics of the type $\theta$-periodic fails to minimize past its period $L$. 

The second part of the proof relies on the basic theory of Jacobi fields and conjugate points. For readers unfamiliar with the subject, refer to \cite[sub-Chapter 4.8]{elem-diff-geo}, or \cite[Chapter 10]{lee-riem-man}. 

Let $\mu$ be such that $R>1$, and let us consider an initial condition $\dot{\theta}(0) = 0$. We will show that $\gamma(L)$ is conjugate to $\gamma(0)$ along $\gamma(t)$, thus $\gamma(t)$ is not minimizing past $\gamma(L)$. To do so, we will construct a killing vector field, which by general theory is a Jacobi field when restricted to a geodesic. By \cite[Section 3]{RM-ABD}, a vector field $K$ is a killing vector field if and only if its momentum function $P_K$ Poisson commutes with the Hamiltonian $H_{sR}$, i.e., $\{H_{sR},P_K\} = 0$. It follows that $\frac{\partial}{\partial x}$ and $\frac{\partial}{\partial y}$ are killing vector fields since $H_{sR}$ is invariant under the Hamiltonian action of $\R^2$, as we discussed above. 

Consider the two Jacobi fields on $\gamma(t)$: 
\begin{align*}
    W_1(t) &= \cos  (\theta_0) \dfrac{\partial }{\partial x} + \sin (\theta_0) \dfrac{\partial}{\partial y}  \; \; \text{ restricted to } \gamma, \;  \text{and}\\
    W_2(t) & = \dot{\gamma}(t). 
\end{align*}
 We see that at $t = kL$, we have $W_2(kL) = X_u(\gamma(kL)) = \cos(\theta_0) \frac{\partial}{\partial x} + \sin(\theta_0) \frac{\partial}{\partial y}$. Therefore, $W_1(0) = W_1(L) = W_2(0) = W_2(L)$. Thus the Jacobi Field $J = W_1 - W_2 $ vanishes at $t=0$ and $t=L$. Moreover, $J$ is not trivial since $\dot{\theta}(t) \neq 0$ along $0 < t < L/2$ and $L/2 < 0 < L$. At the time $t= L/2$, we have that 
 \begin{equation*}
     W_2(L/2) =  \cos(\theta(L/2))\dfrac{\partial}{\partial x} + \sin(\theta( L/2)) \dfrac{\partial}{\partial y}  \neq W_1(L/2).
 \end{equation*} We just shown that $\gamma(L)$ is a conjugate point and fails to minimize beyond $t = L$. 
\end{proof}

\subsection{Proof of Theorem \ref{main:the-1}} \label{Sec:proof-the-main-2}

\begin{proof}
Proposition \ref{Prop: Heteroclinic geod minimizing } shows that the geodesic of type line and heteroclinic geodesics are metric lines. Proposition \ref{prop:cut-time} implies that geodesics of type $\theta$-periodic do not qualify as metric lines since the $\theta$-period $L(\mu)$ is an upper-bound for its cut time. 

Therefore, we conclude that the metric lines in $\SE2$ are precisely the geodesics of the type line and heteroclinic.     
\end{proof}

\section{Conclusion and Future Work}

Conclusions: We studied the \sR geodesic flow on the Special Euclidean group by performing the symplectic reduction of $T^*\SE2$ by $\R^2$, where the metabelian structure of $\SE2$ plays a primary role. We classified \sR geodesics according to their reduced dynamics and selected the candidates to be metric lines. We solved the Hamilton-Jacobi equation and found a calibration function for every geodesic. However, the calibration function was not globally defined. Finding a global calibration function proved that heteroclinic and line geodesics are metric lines. We show that $\theta$-periodic geodesic are not metric lines by providing an upper bound for the cut time.

In the future, we hope to work: 
\begin{enumerate}
    \item Extend our characterizations of the metric lines and periodic geodesics to $\Euclid3$, or more generally to $\Euclidn$.
    \item Study the eigenvalue problems and fundamental solutions of the \sR Laplace, Heat, and Schrödinger operators on $\SE 2$.  
\end{enumerate}

\appendix

\section{Ma\~{n}e's critical value}\label{ape:mane}
In this section, we will show that the calibration function defined in Proposition \ref{prop:glo-cal-fun} corresponds to the Ma\~{n}e's critical value. The Ma\~{n}e's critical value is a fundamental concept from KAM theory. We will follow the approach from \cite{aub-Ham-eq} by A. Figalli and L. Rifford. For a more extensive explanation, refer to \cite{fathi2008weak}. 

Let $M$ be a smooth, connected, and compact Riemannian manifold without boundary. Let $H: T^*M \to \R$ be a Hamiltonian function and $||\cdot||^*_x$ be the norm on $T^*_xM$. We say that $H$ is a \textbf{Tomelli Hamiltonian} if the following conditions are satisfied:

\begin{itemize}
    \item Superlinear growth: For every $K \geq 0$ there exist a constant $C^*(K)$ such that 
    $$ H(x,p) \geq K ||p||^*_x - C^*(K) \quad \text{for all}\; (p,x) \in T^*M. $$

    \item Uniform convexity: For every $(p,x)$ in $T^*M$, the second derivative along the fiber $\frac{\partial^2}{\partial p^2}$ is positive definite. 
\end{itemize}

Notice that reduced Hamiltonian $H_{\mu}$, given by equation \eqref{eq:red-ham}, is a Tonelli Hamiltonian for all value $\mu$ in $\R^2$. The Ma\~{n}e's critical value can be defined as follows.

\begin{defi}
    We say that a real number $c[H]$ is the Ma\~{n}e's critical value of a Hamiltonian function $H$, if  $c[H]$ is the infimum of the values $c$ in $\R$ for which there exist a function $S: M \to \R$ of class $C^1$ satisfying 
    \begin{equation}\label{eq:mane-def}
     H(dS,x) \leq c \quad \text{for all}\; x \in M.   
    \end{equation}

\end{defi}

If the Tonelli Hamiltonian corresponds to a mechanical system, i.e., $H$ has the form
$$ H(p,x) = \frac{1}{2} ||p||^2_* + V(x) \quad \text{for all}\; (p,x) \in T^*M, $$
where $V:M \to \R$ is a potential function of class $C^k$ for $k\geq 2$. Then, it is easy to check that 
$$c[H] = \max_{x \in M} V(x). $$

In the case of a reduced Hamiltonian $H_{\mu}$, we have $\frac{R^2}{2} = c[H_{\mu}]$. The following proposition links the Ma\~{n}e's critical value and a calibration function for a curve over an infinite interval. 

\begin{Prop}\label{prp:mane-cri}
     Suppose $S: M \to \R$ is a calibration function for a curve $\gamma(t)$ over an infinite interval. In that case, $S$ satisfies the inequality from equation \eqref{eq:mane-def} where $c$ is equal to Ma\~{n}e's critical value $c[H]$.
\end{Prop}

Refer to \cite[Proposition 4.3.6]{fathi2008weak} for the proof of Proposition \ref{prp:mane-cri}. We conclude that reduced Hamiltonian function $H_{\mu}$ has a globally minimizing solution of action if $H(p(t),\theta(t)) = \frac{R^2}{2}$. It is easy to show that \sR geodesics whose reduce dynamics has energy $H(p(t),\theta(t)) = \frac{R^2}{2}$ are reparametrization of geodesics of the type heteroclinic and line.

\bibliographystyle{plain}
\bibliography{bibliography.bib}

\end{document}